\documentclass[10pt,a4paper]{amsart}

\usepackage{amssymb,amsmath,amsfonts}

\usepackage{graphicx}

\usepackage{mathpazo}
\usepackage[hyperindex,pageanchor]{hyperref}

\numberwithin{equation}{section}
\newtheorem{theorem}{Theorem}[section]

\newtheorem{proposition}[theorem]{Proposition}
\newtheorem{corollary}[theorem]{Corollary}

\theoremstyle{definition}
\newtheorem{definition}[theorem]{Definition}
\theoremstyle{remark}
\newtheorem{remark}[theorem]{Remark}

\newtheorem{question}[theorem]{Question}

\newcommand{\Spec}{\operatorname{Spec}}

\newcommand{\drdepth}{\operatorname{DR-depth}}

\newcommand{\lcd}{\operatorname{lcd}}

\newcommand{\depth}{\operatorname{depth}}

\newcommand{\fm}{\frak{m}}

\newcommand{\fn}{\frak{n}}

\begin{document}

\title[ de Rham depth]
{On a lower bound of de Rham depth of affine cone of a projective space in characteristic zero}

\author[Eghbali]{Majid Eghbali \   }

\address{Department of Mathematics, Tafresh University, Tafresh, 39518--79611, Iran.}
\email{m.eghbali@tafreshu.ac.ir}
\email{m.eghbali@yahoo.com}




\begin{abstract}
Let $Y \subset \mathbb{P}^n_k$ be a proper closed subset in the projective $n$-space over a field $k$ of characteristic zero and let $C(Y)$ be its affine cone. Let $I \subset R=k[x_0, \ldots, x_n]$ be the homogeneous defining
ideal of $Y$ with homogeneous maximal ideal $\fm=(x_0, \ldots, x_n)$. In this paper, we find out a lower bound of de Rham depth of $C(Y)$ in terms of $\depth R/I$.
\end{abstract}

\maketitle

\section{Introduction}

Throughout the paper, let $Y \subset \mathbb{P}^n_k$ be a proper closed subset and
let $I \subset R=k[x_0, \ldots, x_n]$ be the homogeneous defining
ideal of $Y$ over a field $k$ of characteristic zero with homogeneous maximal ideal $\fm=(x_0, \ldots, x_n)$.

Algebraic de Rham cohomology was defined by Grothendieck \cite{Gr66} for locally of finite type scheme $X$ over a field $k$ and smooth over $k$ as the hypercohomology $\mathbb{H}^i(X,\Omega^{\bullet}_{X})$, where $\Omega^{\bullet}_{X}$ is the complex of sheaves of regular differentials on $X$. 
Later in \cite{Ha75}, Hartshorne defined it for schemes of finite type (with arbitrary singularities) embedded as a closed subscheme in a smooth scheme (see Preliminaries for details). One of the interesting applications of algebraic de Rham cohomology is the famous paper of Ogus (\cite{Og}), where he gives a complete answer to the question raised in \cite{Ha68} on the cohomological dimension of $\mathbb{P}^n_k \setminus Y$ in terms of the de Rham cohomology of $Y$. In the mentioned paper, Ogus defines the de Rham depth of $Y$ (see Definition \ref{0.2}) and shows its relationship with the local cohomological dimension $\lcd (\mathbb{P}^n_k, Y)$ (\cite[Theorem 4.1]{Og}). Notice that the de Rham depth of $Y$ always is bounded above by $\dim Y$.

In the present paper, we consider the following question:
\begin{question}\label{Q}
 If $\depth R/I \geq t$, ($t$ is an integer), then what can we say about a lower bound of de Rham depth of the affine cone of $Y$?
\end{question}
In order to answer the aforementioned question, we prove the following result. 

\begin{theorem} (Theorem \ref{3})
Suppose that $\depth R/I =t \geq 3$. Furthermore suppose that
\begin{itemize}
  \item[(a)] the number of consecutive even integers $i$ starting from $2$ with $H^i_{dR}(Y) =k$ is $t-2$,
  and
  \item[(b)] the number of consecutive odd integers $i$ starting from $3$ with $H^i_{dR}(Y) =0$ is $t-3$.
\end{itemize}
Then,
 $\drdepth C(Y) \geq 2t-2$.
\end{theorem}

\section{Preliminaries}

In this section, we review some definitions and auxiliary results on the algebraic de Rham cohomology. We assume that the reader is familiar with the basic concepts as stated in \cite{Ha75}. 
Let $Y$ be a scheme of finite type over a field $k$ of characteristic $0$ which admits an embedding as a closed subscheme
of a smooth scheme $X$. We denote by ${\Omega}^p_X$ the sheaf of $p$-differential forms on $X$ over $k$. With these notations in mind, we define the de Rham complex of $X$ as the complex of sheaves of differential forms
 $${\Omega}^{\bullet}_{X}: \mathcal{O}_X  \rightarrow {\Omega}^{1}_{X} \rightarrow {\Omega}^{2}_{X} \rightarrow \cdots.$$
 De Rham cohomology of $Y$ is defined as hypercohomology of the formal completion $\hat{\Omega}^{\bullet}_{X}$ of the de Rham complex of $X$ along $Y$:
$$H^i_{dR}(Y/k)=\mathbb{H}^i(\hat{X},\hat{\Omega}^{\bullet}_{X}).$$
Neither does it depend on $X$ nor on the chosen embedding of $Y$ (\cite[Section III]{Ha75}). If there is no ambiguity on the field $k$ we use $H^i_{dR}(Y)$ instead of $H^i_{dR}(Y/k)$ . Below, we recall the definition of local de Rham cohomology of $Y$ which is the main tool in our investigation.

\begin{definition}\label{0.1}
Let $A$ be a complete local ring with coefficient field $k$ of
characteristic zero. Let $\pi :R \twoheadrightarrow A$ be a surjection of
$k$-algebras where $R=k[[x_1, \ldots , x_n]]$ for some $n$ and let
$Y \hookrightarrow X$ (where $Y=\Spec A, X=\Spec R$) be the
corresponding closed immersion. Let $y \in Y$ be the closed
point. The (local) de Rham cohomology of $Y$ is defined by
$$ H^i_{y, dR}(Y)=\mathbb{H}^i_{y}(\hat{X},\hat{\Omega}^{\bullet}_{X}),\ \text{\ for\ all\ } i.$$
\end{definition}
By the Strong excision \cite[III. Proposition 3.1]{Ha75} we have 
$$ H^i_{y, dR}(Y) =\mathbb{H}^i_{y}(\hat{X},\hat{\Omega}^{\bullet}_{X}) \simeq H^i_{y, dR}(\Spec \hat{\mathcal{O}}_{X,y})$$ as
$k$-spaces for all $i$.

It is noteworthy to mention that $H^i_{dR}(Y)$ and $ H^i_{y, dR}(Y)$ are finite dimensional $k$-vector spaces for all $i$ (\cite[Section III]{Ha75}).

\begin{definition}\label{0.2}
 The "\emph{de Rham depth} of $Y$" which is abbreviated by $\drdepth Y$
is $\geq d$ if and only if for each (not necessarily closed) point
$y \in Y$,
$$H^i_{y,dR}(Y)=0, \  i < d- \dim \overline{\{y\}},$$
where $d$ is an integer and $\overline{\{y\}}$ denotes the closure
of $\{y\}$.
\end{definition}
Note that always $\drdepth Y \leq \dim Y$. If $Y$ is a local complete intersection, then $\drdepth Y = \dim Y$.
Here, and by our assumption on the scheme $Y \subset \mathbb{P}^n_k$ we may suppose that $y$ is a closed point of $Y$. See \cite[pp. 340]{Og}.

We conclude this section with the following remarks, which are useful in the next section.

\begin{remark} \label{0.8} Let $I \subset R=k[x_0, \ldots, x_n]$ be a homogeneous ideal of a polynomial ring over a field $k$ and $\fm=(x_0, \ldots, x_n)$ be its homogeneous maximal ideal. Note that $k = R/\fm$.  Using a suitable gonflement of $R$ (cf. \cite[Chapter IX, Appendice 2]{Bo}) one can take a faithfully flat local homomorphism $f: R \rightarrow S$ with $f(\fm)=\fn$ from $(R,\fm)$ to a regular local ring
$(S,\fn)$ such that $S/\fn$ is the algebraic closure of $k$. Because of faithfully flatness, $S$ contains a field.
It shows that $\depth R/I=\depth S/IS$.
\end{remark}

\begin{remark} \label{0.9} Suppose that $\mathbb{Q} \subseteq k_0 \subseteq k$, where $k$ is a field extension of a field $k_0$. As cohomology of coherent sheaves commutes with flat base extension, then algebraic (local) de Rham cohomology does not alter under field extensions. So one can reduce questions on the vanishing of de Rham cohomology from the given field $k_0$ to the complex field $\mathbb{C}$. 
\end{remark}

\section{Main Result}

Our aim is to find out a relationship between $\depth R/I$ and $\drdepth C(Y)$, where $Y \subset \mathbb{P}^n_k$ is a proper closed subset of projective $n$-space over a field $k$ of characteristic zero and
 $I \subset R=k[x_0, \ldots, x_n]$ is the homogeneous defining
ideal of $Y$. First we consider the case $Y$ is smooth and next we deal with the general case.

\begin{remark}\label{cone}
  Suppose that $Y$ is a smooth variety of dimension $n$, then $\drdepth Y=n$ (\cite[Remark 2.18]{Og}). But $C(Y)-\{y\}$ is smooth of dimension $n+1$, so its $\drdepth$ is $n+1$, where $y$ is the vertex of affine cone $C(Y)$ of $Y$ corresponds to the maximal irrelevant ideal $\fm$ of $R$.
\end{remark}

We are now ready to answer Question \ref{Q} stated in the Introduction.

\begin{proposition}\label{1}
If $\depth R/I \geq 2$, then $\drdepth C(Y) \geq 2$.  
\end{proposition}

\begin{proof}
To show  $\drdepth C(Y) \geq 2$, we must show that for all $y \in C(Y),\ H^i_{y,dR}(C(Y)) = 0$ for all $i < 2- \dim \overline{\{y\}}$. 
By virtue of \cite[2.15 or 2.16]{Og} one may assume that $y$ is a closed point. Hence, we fix a closed point $y \in C(Y)$ and show that $H^i_{y,dR}(C(Y)) = 0$ for  $i =0,1$.

According to the assumption $\depth R/I \geq 2$, thus $C(Y)-\{y\}$ is connected (\cite[Proposition 2.2]{Ha62}). Therefore, by \cite[pp. 53]{Ha75} we have $H^0_{dR}(C(Y)-\{y\}) \cong k$. Next, consider the long exact sequence 
$$0 \rightarrow H^0_{y,dR}(C(Y)) \rightarrow H^0_{dR}(C(Y)) \rightarrow H^0_{dR}(C(Y)-\{y\}) \rightarrow H^1_{y,dR}(C(Y)) \rightarrow H^1_{dR}(C(Y)), \ (\ast)$$
Notice that $H^0_{dR}(C(Y)) \cong k$ and $H^1_{dR}(C(Y))=0$ because $C(Y)$ is topologically contractible. Hence, from the above exact sequence and the fact that finite dimensional vector spaces with the same dimension are isomorphic, it yields that $H^0_{y,dR}(C(Y))=0= H^1_{y,dR}(C(Y))$.
\end{proof}
It is noteworthy to mention that by virtue of the exact sequence ($\ast$) appeared in the proof of Proposition \ref{1} one may deduce that $\drdepth C(Y) \geq 1$ without any assumptions on $\depth R/I$.

We distinguish the following case (Theorem \ref{2}) from the general case (Theorem \ref{3}) in order to avoid complexity. But before it let us bring a result of Hartshone which is useful in the sequel.

\begin{proposition}\label{Hartshorne} (\cite[Proposition 3.2  page 73.]{Ha75}) Let $Y$ be as above. Then, there are exact sequences
$$0 \rightarrow k \rightarrow H^0_{dR}(Y) \rightarrow H^{1}_{y,dR}(C(Y)) \rightarrow 0$$
and
$$\ldots \rightarrow H^{i-2}_{dR}(Y) \rightarrow H^i_{dR}(Y) \rightarrow H^{i+1}_{y,dR}(C(Y)) \rightarrow H^{i-1}_{dR}(Y) \rightarrow \ldots,$$
where $i \geq 1$ and $y \in C(Y)$ is a vertex.
\end{proposition}

\begin{theorem}\label{2}
Suppose that $\depth R/I \geq 3$ and
 $H^2_{dR}(Y) = k$. 
Then,
 $\drdepth C(Y) \geq 4$.
\end{theorem}

\begin{proof}
Note that the cohomology of $Y$ and $Y_{red}$
are the same. Then we may assume that $Y$ is reduced and by Remarks \ref{0.8} and \ref{0.9} we may suppose that $k=\mathbb{C}$.

We show that $\drdepth C(Y) \geq 4$.  According to what we have seen in the proof of Proposition \ref{1}, it is enough to show that $H^i_{y,dR}(C(Y))=0$ for $i=0,1,2, 3$ for a closed point $y \in C(Y)$. 
As $\depth R/I \geq 3$, by Proposition \ref{1} one has $H^i_{y,dR}(C(Y))=0$ for $i=0,1$. Thus, it remains to show that $H^i_{y,dR}(C(Y))=0$ for $i=2,3$.

Let us give an outline of the proof. First, we show that $H^1_{dR}(Y) = 0$ and then using Proposition \ref{Hartshorne}, we will show that $H^i_{y,dR}(C(Y))=0$ for $i=2,3$.

\textbf{Proof of $H^1_{dR}(Y) = 0$:} 

As $H^1_{dR}(Y) \cong H^1(Y^{an}, \mathbb{C})$, To prove the claim, using the universal coefficient Theorem it is enough to prove that $H^1(Y^{an}, \mathbb{Z})=0$. Since $Y^{an}$ is a reduced complex analytic space, one has the exact sequence of sheaves 
$$ 0 \longrightarrow \mathbb{Z}_{Y^{an}} \longrightarrow \mathcal{O}_{Y^{an}} \longrightarrow \mathcal{O}^{\ast}_{Y^{an}} \longrightarrow 0,$$
where $\mathbb{Z}_{Y^{an}}$ is the constant sheaf, $\mathcal{O}_{Y^{an}}$ is the structure sheaf, and $\mathcal{O}^{\ast}_{Y^{an}}$ is the sheaf
of invertible elements of $\mathcal{O}_{Y^{an}}$ under multiplication. The cohomology sequence of this short exact sequence implies the exact sequence of groups
$$ 0 \longrightarrow \mathbb{Z} \longrightarrow \mathbb{C} \longrightarrow \mathbb{C}^{\ast} \longrightarrow 0,$$
 because the global holomorphic functions are constant, \cite[Page 446]{Ha77}. Then we have the injective map 
 $ H^1(Y^{an}, \mathbb{Z}_{Y^{an}}) \rightarrow H^1(Y^{an}, \mathcal{O}_{Y^{an}})$. On the other hand, note that 
$$H^1(Y^{an}, \mathcal{O}_{Y^{an}}) \cong H^1(Y, \mathcal{O}_{Y}) \cong (H^2_{\fm}(R/I))_0=0,$$
because $\depth R/I \geq 3$. To this end, note that the first isomorphism follows from \cite{Serre56}. Hence, $ H^1(Y^{an}, \mathbb{Z}) \cong H^1(Y^{an}, \mathbb{Z}_{Y^{an}}) $ must be zero.

Next, exploiting Proposition \ref{Hartshorne} we have the following exact sequence 
$$0= H^{1}_{dR}(Y)  \rightarrow H^{2}_{y,dR}(C(Y)) \rightarrow H^{0}_{dR}(Y) \rightarrow H^{2}_{dR}(Y) \rightarrow H^{3}_{y,dR}(C(Y)) \rightarrow H^{1}_{dR}(Y)=0,$$
which that $H^i_{y,dR}(C(Y))=0$ for $i=2,3$. To this end, note that by the assumption $H^{2}_{dR}(Y)=k$ and because of the vanishing of $H^1_{y,dR}(C(Y))$ and the Proposition \ref{Hartshorne} we have $H^{0}_{dR}(Y)=k$. 
\end{proof}

\begin{theorem}\label{3}
Suppose that $\depth R/I =t \geq 3$. Furthermore suppose that
\begin{itemize}
  \item[(a)] the number of consecutive even integers $i$ starting from $2$ with $H^i_{dR}(Y) =k$ is $t-2$,
  and
  \item[(b)] the number of consecutive odd integers $i$ starting from $3$ with $H^i_{dR}(Y) =0$ is $t-3$.
\end{itemize}
Then,
 $\drdepth C(Y) \geq 2t-2$.
\end{theorem}

\begin{proof} We prove the claim by induction on $t$. The case $t=3$ has been shown in Theorem \ref{2}.
Now assume that $t >3$ and the claim is true for $t-1$, i.e. $\drdepth C(Y) \geq 2t-4$.
In order to show that $\drdepth C(Y) \geq 2t-2$, according to what we have seen in the proof of Proposition \ref{1}, it is enough to show that $H^i_{y,dR}(C(Y))=0$ for $i=0, \cdots, 2t-3$ for a closed point $y \in C(Y)$. 
As $\depth R/I = t > t-1$, by the induction hypothesis one has $H^i_{y,dR}(C(Y))=0$ for $i=0, \cdots, 2t-5$. Thus, it remains to show that $H^i_{y,dR}(C(Y))=0$ for $i=2t-4, 2t-3$, where the proof is completely similar to what we have seen in the proof of Theorem \ref{2}.
\end{proof}

We end this section by an instant byproduct from Theorem \ref{3} in conjunction with \cite[Corollary 7.3, page 86]{Ha75}:

\begin{corollary}\label{2.1}
Let $Y \subset \mathbb{P}^n_k$ be a set-theoretic complete intersection projective variety over a field $k$ of characteristic zero and
let $I \subset R=k[x_0, \ldots, x_n]$ be the homogeneous defining
ideal of $Y$. If $\depth R/I =t \geq 3$ and dimension of $Y$ is $2t-3$, then
 $\drdepth C(Y) \geq 2t-2$.
\end{corollary}

\subsection*{Acknowledgment}The author would like to thank Professor Robin Hartshorne for helpful comments on the earlier version of the paper. We are also grateful to Professor Arthur Ogus for a discussion on the de Rham depth.


\begin{thebibliography}{10}


\bibitem {Bo}
N. Bourbaki, \emph{\'{E}l\'{e}ments de math\'{e}matique. Alg\'{e}bre commutative.}, Chapters 8 et 9, Springer, (2008).
 
\bibitem{Gr66}
 A. Grothendieck, \emph{On the De Rham cohomology of algebraic varieties.}, Publ. Math. I.H.E.S., {\bf 29}, 95--103 (1966).

\bibitem {Ha62}
R. Hartshorne, \emph{Complete Intersections and Connectedness.}, Amer. J. of Math. {\bf 84}(3), 497--508 (1962).

\bibitem {Ha68}
R. Hartshorne, \emph{Cohomological dimension of algebraic varieties.}, Ann. of Math. {\bf}, 403--450 (1968).

\bibitem{Ha75}
R. Hartshorne, \emph{On the de Rham cohomology of algebraic varieties}, Inst. Hautes Etudes Sci. ´
Publ. Math {\bf 45}, 5--99 (1975).

\bibitem {Ha77}
R. Hartshorne, \emph{Algebraic Geometry.}, Springer Graduate Texts, (1977).

 \bibitem{Og}
A. Ogus, \emph{Local cohomological dimension of algebraic varieties.}, 
 Ann. of Math. ,  {\bf 98}(2), 327--365 (1973).
 
  \bibitem{Serre56}
J. P. Serre, \emph{Geometrie algebrique et geometrie analytique.}, 
Ann. Inst. Fourier,  {\bf 6}, 1--42 (1956).
  
\end{thebibliography}
\end{document}